\documentclass[12pt]{amsart}
\usepackage{amsthm}
\usepackage{amssymb,amsfonts}
\usepackage{amsmath}
\usepackage{verbatim,enumerate,url,graphicx}
\usepackage[a4paper,margin=1in]{geometry}
\usepackage{bm}

\renewcommand{\thefootnote}{\fnsymbol{footnote}}

\newtheorem{theorem}{Theorem}[section]
\newtheorem{lemma}[theorem]{Lemma}

\newtheorem{proposition}[theorem]{Proposition}
\newtheorem{remark}[theorem]{Remark}

\newtheorem{example}[theorem]{Example}

\newtheorem*{example*}{Example}

\newtheorem*{theorem*}{Theorem}

\newtheorem*{remark*}{Remark}

\newtheorem{corollary}[theorem]{Corollary}
\newtheorem*{corollary*}{Corollary}

\newtheorem*{definition*}{Definition}

\newtheorem*{notation*}{Notation}

\numberwithin{equation}{section}

\gdef\myletter{}

\let\savetheequation\theequation
\def\theequation{\savetheequation\myletter}

\newcommand{\CC}{{\mathbb C}}

\newcommand{\RR}{{\mathbb R}}

\newcommand{\NN}{{\mathbb N}}

\def\fl{\mathcal{F}}

\newcommand{\calF}{\mathcal{F}}

\newcommand{\calC}{\mathcal{C}}

\def \b0{{\bf 0}}

\long\def\symbolfootnote[#1]#2{\begingroup%
\def\thefootnote{\fnsymbol{footnote}}\footnote[#1]{#2}\endgroup}

\begin{document}

\title[Properties of fast Leja points]{On the approximation properties of fast Leja points}
\author{Sione Ma`u}

\address{Department of Mathematics,
University of Auckland,
Auckland, NZ}
\email{s.mau@auckland.ac.nz}

\begin{abstract}
Fast Leja points on an interval are points constructed using a discrete modification of the algorithm for constructing Leja points.  Not much about fast Leja points has been proven theoretically. We present an asymptotic property of a triangular interpolation array, and under the assumption that fast Leja points satisfy this property, we prove that they are  good  for Lagrange interpolation.   
\end{abstract}

\maketitle

\section{Introduction}\label{sec:intro}

  Recall that for a compact set $K\subset\CC$, Leja points are constructed inductively as follows:
\begin{enumerate}[1.]
\item Choose $z_0\in K$.
\item Given $z_0,\ldots,z_{n-1}$, choose $z_n\in K$ such that $$|p_n(z_n)|=\|p_n\|_K:=\sup_{z\in K}|p_n(z)|,$$ where $p_n(z)=(z-z_0)\cdots(z-z_{n-1}).$
\end{enumerate}

Leja points are found in step 2 by optimization, and computations become cumbersome for large $n$.  One way to get around this is to discretize $K$ with a (weakly) admissible mesh \cite{cl:uniform}, from which so-called \emph{discrete Leja points} can be constructed by numerical linear algebra \cite{bdMsv:computing}.  The theory of discrete Leja points on weakly admissible meshes has been studied a lot in higher dimensions; see \cite{mv:weakly} for more resources and a list of references.

Fast Leja points were introduced in \cite{bcr:fast}  as another discrete version of Leja points.  Unlike points constructed using weakly admissible meshes, a preliminary discretization is not required.  The construction of fast Leja points proceeds in tandem with an adaptive discretization of $K$ by so-called \emph{candidate points}.  

Fast Leja points are defined on Jordan arcs or curves in $\CC$.  The prototypical Jordan arc is the interval $[0,1]$, where the fast Leja points $T_n$ and candidate points $S_n$ are constructed inductively as follows.
\begin{enumerate}[1.]
\item Let $t_0=0$, $t_1=1$, and choose $s_1\in(0,1)$.
\item Given the sets $$T_{n-1}=\{t_0,\ldots,t_{n-1} \}, \ S_{n-1}=\{s_1,\ldots,s_{n-1}\}$$ and polynomial $p_n(z)=(z-t_0)\cdots(z-t_{n-1})$, choose $s_j\in S_{n-1}$ such that
\begin{equation*}
|p_n(s_k)|=\max\{|p_n(s_j)|\colon j=1,\ldots,n-1\}.
\end{equation*}
\item Choose $a,b\in(0,1)$ such that $[a,b]\cap\{s_1,\ldots,s_n\}=\{s_k\}$.
\item Put $t_n:=s_k$, $s_k:=a$ and $s_{n}:=b$; then $T_n=\{t_0,\ldots,t_n\}$ and $S_{n}=\{s_1,\ldots,s_n\}$.
\end{enumerate}
 Each step of the algorithm yields the fast Leja points $T_n=T_{n-1}\cup\{t_n\}$, as well as the candidate points $S_n$ that interlace the fast Leja points.  At step $n$, a point $s_k$ is moved from $S_{n-1}$ to $T_n$ and replaced by 2 points on either side.  There is some flexibility in choosing the new candidate points $a,b$ of $S_n$; a standard choice is to take the midpoints of the two intervals joining $s_k$ with the adjacent fast Leja points in $T_{n-1}$ to the left and right.

If $K$ is a Jordan arc, it has a continuous parametrization $z\colon[0,1]\to K$. Define $z_0=z(0)$, $z_1=z(1)$, $m_1=z(s_1)$ where $s_1\in(0,1)$.  The criterion (step 2) for adding  $z_n$ to $\{z_0,\ldots,z_{n-1}\}$ from the candidates $\{m_1,\ldots,m_n\}$ is $$|p_n(z_n)|=\max\{|p_n(m_j)|\colon j=1,\ldots,n-1\}.$$  Steps 3--4 are exactly the same, done in terms of the parameters $t_j$ and $s_k$ where $z_j= z(t_j)$ and $m_k=z(s_k)$.

If $K$ is a Jordan curve, it has a parametrization with $z(0)=z(1)$.  Let $z_0$ be this point and let $m_0=z(s_0)$ for some choice of $s_0\in(0,1)$.  Then run the algorithm as before.  In this case the number of fast Leja points and interlacing candidate points is the same at each step.


Although fast Leja points have been around for over 20 years, there does not seem to have been much rigorous study of their properties.  In this paper, we will consider fast Leja points on a real interval, where the candidate points are chosen to be the midpoints of the intervals between adjacent fast Leja points.

An open problem is to prove that fast Leja points are good for polynomial interpolation.  To do (Lagrange) polynomial interpolation, one specifies a  triangular interpolation array $\calC$, which is a set of points and indices $\{a_{nj}\}_{\substack{j=0,\ldots,n;\; n=1,2,\ldots}}\subset  K$, where points are distinct at each stage: $a_{nj}\neq a_{nk}$ if $j\neq k$, for each $n$. Given a function $f$, let $L_nf$ be the unique polynomial of degree $n$ such that $$f(a_{nj})=L_nf(a_{nj}) \hbox{ for each } j=0,\ldots,n.$$
 Then the array $\calC$ is good for polynomial interpolation on $K$ if $L_nf$ converges uniformly to $f$ on $K$ as $n\to\infty$ (written $L_nf\rightrightarrows f$), for all $f$ analytic in a neighborhood of $K$.  

Conditions that give good interpolation arrays for a compact set $K$ were described in \cite{bloomboschristensenlev:polynomial}; one of these conditions involves the \emph{transfinite diameter $d(K)$}  (cf. Theorem \ref{thm:bbcl}, property (2)).  The main theorem of this paper, Theorem \ref{thm:d}, connects this condition to  a certain asymptotic distribution property (Property ($\star$), see Section \ref{sec:star}) of the interpolation array.  It is easy to construct many examples of arrays that satisfy this property; indeed,  in Theorems \ref{thm:6}--\ref{thm:generalthm} some rather flexible bounds on asymptotic behaviour of the local density of points of the array are derived that guarantee Property ($\star$).  Verifying these bounds for fast Leja points is yet to be done, but I plan to return to this in a future work.

The outline of the paper is as follows.  Section \ref{sec:sup} carries out preliminary calculations relating the sup norm of a polynomial on an interval to its value at the midpoint of the interval.  In Section \ref{sec:star}, Property ($\star$) is defined and studied.   Then the main theorem is proved in Section \ref{sec:approx}, summarized below:
\begin{theorem*}[{\rm cf. Theorem \ref{thm:d}}]
Let $\fl$ denote the fast Leja points on an interval.  Suppose Property ($\star$) holds for $\fl$.  Then $\fl$ is good for polynomial interpolation.
\end{theorem*}  Finally, in Section 5 we make some closing remarks.

\section{Estimating the sup norm}\label{sec:sup}

For convenience of calculation in this section,  we will translate our interval so that the point of interest is the origin.  The setup is as follows: let $n_1,n_2\in\NN$, and put  $n:=n_1+n_2+2$.  Let $\{\zeta_j\}_{j=1}^{n_1},\{\eta_j\}_{j=1}^{n_2}$ be sequences of real numbers such that
$$\zeta_1<\zeta_2<\cdots<\zeta_{n_1}<-\epsilon<\epsilon<\eta_1<\eta_2<\cdots<\eta_{n_2}.$$
We will also write the full collection of points as $\{z_j\}_{j=1}^n$, i.e.,  
$$z_1=\zeta_1, \ \ldots, \ z_{n_1+1}=-\epsilon,  \  z_{n_1+2}=\epsilon,\ z_{n_1+3}=\eta_1,\ \ldots,\ z_{n}=\eta_{n_2}.$$

\begin{lemma}\label{lem:2}
Let $\displaystyle p(z):=\prod_{j=1}^n (z-z_j)$, 
$$\begin{aligned}
p_1(z)&:=(z^2-\epsilon^2)\prod_{j=1}^{n_1}(z-\zeta_j)=:(z^2-\epsilon^2)q_1(z),\\
 p_2(z)&:=(z^2-\epsilon^2)\prod_{j=1}^{n_2}(z-\eta_j)=:(z^2-\epsilon^2)q_2(z),
\end{aligned}$$
 so that
$p=p_1q_2=q_1p_2$. 

Let $m\in[-\epsilon,\epsilon]=I$ be the point for which $|p(m)|=\|p\|_I$.  Then
\begin{equation}
|q_2(m)| < \exp\left(\epsilon\sum_j 1/|\eta_j| \right)|q_2(0)|\hbox{ and } |p_1(m)|<  \epsilon|p_1'(\epsilon)|, \hbox{  if $m\in(0,\epsilon]$,}   \label{eqn:lem2.1} 
\end{equation}
  and 
\begin{equation}
|q_1(m)| < \exp\left(\epsilon\sum_j 1/|\eta_j| \right) |q_1(0)| \hbox{ and } |p_2(m)| <\epsilon|p_2'(-\epsilon)|,  \hbox{  if $m\in[-\epsilon,0)$.}  \label{eqn:lem2.2}
\end{equation}
\end{lemma}

\begin{proof}
We prove \eqref{eqn:lem2.2};  the proof of \eqref{eqn:lem2.1} is similar.

Suppose $m\in[-\epsilon,0)$.  Since $m<0<\eta_j$ for all $j$, 
$$\begin{aligned}
\frac{|q_1(m)|}{|q_1(0)|}= \frac{\prod_j(|m|+|\eta_j|)}{\prod_j |\eta_j|} &= \prod_j \left(1+\frac{|m|}{|\eta_j|}\right)  \\  &\leq \exp\left(\sum_j |m|/|\eta_j| \right)
\leq\exp\left(\sum_j \epsilon/|\eta_j| \right).
\end{aligned}$$  

Next, by the mean-value theorem, $|p_2(m)|=|p_2'(c)|\cdot|m|$ for some $c\in(-\epsilon,m)$.  By elementary calculus, $|p_2'(z)|\leq |p_2'(-\epsilon)|$ on $[-\epsilon,m]$ if $|p_2'|$ is decreasing to zero on $[-\epsilon,m]$.   

We claim that, indeed, $|p_2'|$ is decreasing to zero on $[-\epsilon,m]$.  Suppose not.  Then $p_2'$ will have a critical point $a\in[-\epsilon,m]$.  However, $p_2'(m)=0$ and  the zeros of $p_2'$ must interlace the zeros of $p_2$.  Therefore, apart from $m$, all other zeros of $p_2'$ lie outside $[-\epsilon,\epsilon]$ and to the right of this interval.   On the other hand, $a<m$ and $p_2''(a)=0$, which implies that $a$ lies between two zeros of $p_2'$.  In particular, there is a zero of $p_2'$ less than $m$, a contradiction.

So the claim holds and $|p_2(m)|<|p_2'(c)|\cdot|m|<\epsilon|p_2'(-\epsilon)|.$  This completes the proof of \eqref{eqn:lem2.2}.
\end{proof}

\begin{proposition} \label{prop:3}
Let  $I=[-\epsilon.\epsilon]$.  Then 
$$
\frac{\|p\|_{I}}{|p(0)|} \leq 2\exp\left( \epsilon\sum_{j=1}^n \frac{1}{|\zeta_j|}  \right).
$$
\end{proposition}

\begin{proof} 
Let $m\in [-\epsilon,\epsilon]$ satisfy $|p(m)|=\|p\|_I$. When $m=0$ the inequality is trivial.  

Suppose $m\in(0,\epsilon]$.  Then by Lemma \ref{lem:2}, equation \eqref{eqn:lem2.1}, 
$$
\frac{|p(m)|}{|p(0)|} = \frac{|p_1(m)q_2(m)|}{|p_1(0)q_2(0)|}\leq \exp\left(\epsilon\sum_j1/|\eta_j|\right)\epsilon\frac{|p_1'(\epsilon)|}{|p_1(0)|} .
$$
Now $|\zeta_j-\epsilon|= |\zeta_j|+\epsilon$ since $\zeta_j<0<\epsilon$, so  $|p_1'(\epsilon)| = 2\epsilon\prod_{j=1}^{n_1}(|\epsilon|+|\zeta_j|)$.  Using this together with $|p_1(0)|=\epsilon^2\prod_{j=1}^{n_1}|\zeta_j|$, we have  
$$\begin{aligned}
\frac{|p(m)|}{|p(0)|} \leq \exp\left(\epsilon\sum_{j=1}^{n_2} 1/|\eta_j|\right) 2\prod_{j=1}^{n_1}\left(1+\frac{|\epsilon|}{|\zeta_j|}\right) 
& \leq  2\exp\left(\epsilon\sum_{j=1}^{n_2} 1/|\eta_j|\right)\exp\left(\epsilon\sum_{j=1}^{n_1}  1/|\zeta_j|\right). \\
&\leq 2\exp\left(\epsilon\sum_{j=1}^n \frac{1}{|\zeta_j|}\right).
\end{aligned}$$
When $m\in(-\epsilon,0]$ we get the same estimate by the same argument, this time using  \eqref{eqn:lem2.2}.
\end{proof}

\section{An asymptotic property}\label{sec:star}

Consider a triangular array $\calC = \coprod_{n=1}^{\infty}\calC_n$, where $\calC_n=\{\zeta_{nj}\}_{j=1}^n \subset\RR$ are the points of the arrray at stage $n$, and assume the points of $\calC_n$ are distinct and listed in increasing order: $\zeta_{nj}<\zeta_{nk}$ if $j<k$. Let
$$
m_{nj}:= \tfrac{1}{2}(\zeta_{nj}+\zeta_{n(j+1)}) \quad  (j\in\{1,\ldots,n-1\})
$$
be the midpoint of the interval joining adjacent points.  Define
$$
s_n(j) := |m_{nj} -\zeta_{nj}|, \quad  
H_n(j) := \left(\frac{1}{n}\sum_{k=1}^n \frac{1}{|m_{nj} -\zeta_{nk}|} \right)^{-1}. 
$$
We are interested in whether or not  $\calC$ has the following asymptotic property:
$$
  \lim_{n\to\infty}\max\left\{ \frac{s_n(j)}{H_n(j)}\colon  j=1,\ldots,n-1\right\}  =   0 . \eqno(\star)
$$

A sequence $\{\zeta_j\}_{j=1}^{\infty}$ is also defined to have Property ($\star$) if its corresponding triangular array has Property ($\star)$.  Here, the $n$-th stage of the corresponding array is constructed by listing the first $n$ points of the sequence in increasing order.

\begin{example}\rm 
We illustrate Property ($\star$) for a simple sequence that converges to the uniform distribution on $[0,1]$.    
Let $$\calC=\{\frac{m}{2^{k}}\colon m,k\in\NN, m\leq 2^k\},$$ be the fractions in the interval $[0,1]$ whose denominator is a power of 2, and suppose we list the elements of $\calC$ sequentially by the size of the denominator, then numerator, referred to simplest form, i.e., 
$$
0,1,\tfrac{1}{2},\tfrac{1}{4},\tfrac{3}{4},\tfrac{1}{8},\tfrac{3}{8},\tfrac{5}{8},\tfrac{7}{8},\ldots
$$
To give a specific calculation: when $n=9$ and $j=7$ we have $\zeta_{9,7}=\dfrac{3}{4}$, $m_{9,7}=\dfrac{13}{16}$,  $$s_9(7)=\frac{1}{16},  \hbox{ and } \frac{1}{H_9(7)}=\frac{16}{9}\left(\frac{1}{13}+\frac{1}{11}+\frac{1}{9}+\frac{1}{7} +\frac{1}{5}+\frac{1}{3}+1+1+\frac{1}{3}\right)=\frac{2370064}{405405}.$$
For general $n$, suppose  $n\in[2^k+1,2^{k+1}]$ for some $k$.  Given $j<n$, either $s_n(j)=2^{-k}$ or $s_n(j)=2^{-(k+1)}$. In either case, $s_n=O(\frac{1}{n})$.  We also have $\frac{1}{H_n(j)} = O({\log(n)})$,  since these finite sums can be estimated by a multiple of the first $n$ terms of the harmonic series.  Altogether, 
$$\frac{s_n(j)}{H_n(j)} = O(\frac{\log(n)}{n}) \longrightarrow 0 \hbox{ as } n\to\infty.$$ 
The big-O estimates can be made independent of $j$, so Property $(\star)$ holds for $\calC$.
\end{example}

Basic estimates as above can be done very generally to yield Property ($\star$).

\begin{theorem}\label{thm:6}
Let  $\{b_{nj}\}$ be a triangular array of points in $[0,1]$.  Suppose there exist positive constants $\alpha_1,\alpha_2,B_1,B_2$ with the following properties:
\begin{enumerate}
\item $0<\alpha_2\leq\alpha_1<1+\alpha_2$.

\item  There exists $n_0\in\NN$ such that for every integer $n\geq n_0$, 
\begin{equation}\label{eqn:B1B2}
 B_1\Bigl|\frac{j-k}{n}\Bigr|^{\alpha_1} \leq |b_{nj}-b_{nk}| \leq B_2\Bigl|\frac{j-k}{n}\Bigr|^{\alpha_2},  \hbox{ whenever }  j,k\in\{1,\ldots,n\}.
\end{equation} \end{enumerate}

Then Property ($\star$) holds for $\{b_{nj}\}$.
\end{theorem}

\begin{proof}
For each $n\in\NN$ define
$$
j_n:=\text{Arg}\max_{j\in\{1,\ldots,n\}  } \frac{s_n(j)}{H_n(j)} .
$$
To verify Property ($\star$) it suffices to show that ${s_n(j_n)}/{H_n(j_n)}\to 0$ as $n\to\infty$.  For the purpose of contradiction, suppose not. 
Then $\displaystyle\limsup_{n\to\infty} \frac{s_n(j_n)}{H_n(j_n)}>0$, hence there exists $\delta>0$ and a subsequence $\{n_l\}_{l=1}^{\infty}\subset\NN$ such that
\begin{equation}\label{eqn:sn/Hn}
\frac{s_{n_l}(j_{n_l})}{ H_{n_l}(j_{n_l})} \to\delta    \hbox{ as } l\to \infty. 
\end{equation}
Now 
$$
\begin{aligned}
\frac{1}{H_{n_l}(j_{n_l})} \ = \ \frac{1}{n_l}\sum_{k=1}^{n_l} \frac{1}{|b_{n_lk}-m_{n_lj_{n_l}}|}   
\ &= \  \frac{1}{n_l}\sum_{k=1}^{j_{n_l}} \frac{1}{m_{n_lj_{n_l}} - b_{n_lk}}  \ + \    \frac{1}{n_l}\sum_{k=j_{n_l}+1}^{n_l} \frac{1}{ b_{n_lk} - m_{n_lj_{n_l}} } \\ 
& =: \  \frac{1}{H_{n_l}^{(1)}} \  + \   \frac{1}{H_{n_l}^{(2)}}.
\end{aligned}$$
We concentrate on the second term.  First,
$$\begin{aligned}
\frac{1}{H_{n_l}^{(2)}} = \frac{1}{n_l}\sum_{k=j_{n_l}+1}^{n_l} \frac{1}{ b_{n_lk} - m_{n_lj_{n_l}} } 
& \ \leq \   \frac{1}{n_ls_{n_l}(j_{n_l})}  \ +  \   \frac{1}{n_l}\sum_{k=j_{n_l}+2}^{n_l} \frac{2}{ b_{n_lk} - b_{n_lj_{n_l}} }  \\
&  \  \leq \   \frac{1}{n_ls_{n_l}(j_{n_l})}  \ +  \ 
\frac{2}{B_1}\sum_{k=j_{n_l}+2}^{n_l} \left(\frac{1}{ (k- j)/n_l }\right)^{\alpha_1} \frac{1}{n_l} \\
& \ \leq \     \frac{1}{n_ls_{n_l}(j_{n_l})}\  + \  \frac{2}{B_1} \int_{\frac{1}{n_l}}^{1-\frac{j_{n_l}}{n_l}} \frac{dx}{x^{\alpha_1}} \\
&\leq   \frac{1}{n_ls_{n_l}(j_{n_l})}  \ + \  \frac{2}{B_1}  \int_{\frac{1}{n_l}}^{1} \frac{dx}{x^{\alpha_1}}
\end{aligned}$$
where we use the lower estimate of equation \eqref{eqn:B1B2} in the second line of the calculation of $1/H_{n_l}^{(2)}$.

 We will consider different cases of $\alpha_1$.  
When $\alpha_1<1$, we estimate 
$$
\frac{1}{H_{n_l}^{(2)}}   \leq   \frac{1}{n_ls_{n_l}(j_{n_l})}    + \frac{2}{B_1} \int_0^1 \frac{dx}{x^{\alpha_1}}  \leq \frac{1}{n_ls_{n_l}(j_{n_l})} + C
$$
where $C$ is a  positive constant. Also, $s_{n_l}(j_{n_l})< B_2n_l^{-\alpha_2}$ from the definition of $s_{n_l}(j_{n_l})$ and the upper estimate in \eqref{eqn:B1B2}.
   Hence 
$$
\frac{s_{n_l(j_{n_l})}}{H_{n_l}^{(2)}}\leq \frac{1}{n_l} + CB_2n_l^{-\alpha_2}    \to 0    \hbox{ as } l\to\infty.
$$

When $\alpha_1=1$, we use the following estimate:
$$
\frac{1}{H_{n_l}^{(2)}}  \ \leq \   \frac{1}{n_ls_{n_l}(j_{n_l})}   \  + \   \frac{2}{B_1} \int_{\tfrac{1}{n_l}}^1 \frac{dx}{x} 
 \ = \    \frac{1}{n_ls_{n_l}(j_{n_l})} \  +  \   \frac{2}{B_1 }\log(n_l)  
$$ and so $$\frac{s_{n_l(j_{n_l})}}{H_{n_l}^{(2)}} \leq \dfrac{1}{n_l} + C\log(n_l)n_l^{-\alpha_2}\to 0 \hbox{ as } l\to\infty.$$

When $\alpha_1>1$, 
$$\frac{1}{H_{n_l}^{(2)}} \leq  \frac{1}{n_ls_{n_l}(j_{n_l})} +  \frac{2}{B_1}\int_{\frac{1}{n_l}}^1 \frac{dx}{x^{\alpha_1}} 
= \frac{1}{n_ls_{n_l}(j_{n_l})} + \frac{2}{B_1(\alpha_1-1)}(n_l^{\alpha_1-1}-1)
$$ from which it follows that 
$$
\frac{s_{n_l(j_{n_l})}}{H_{n_l}^{(2)}}  \leq \dfrac{1}{n_l} +  \frac{2}{B_1(\alpha_1-1)}n_l^{\alpha_1-1-\alpha_2}\to 0 \hbox{ as } l\to\infty.
$$

Hence in all cases, $\dfrac{s_{n_l(j_{n_l})}}{H_{n_l}^{(2)}}\to 0$.  Similar methods may be used to prove $\dfrac{s_{n_l(j_{n_l})}}{H_{n_l}^{(1)}}\to 0$ also.  Altogether 
$$
\frac{s_{n_l}(j_{n_l})}{ H_{n_l}(j_{n_l})} \to 0 \hbox{ as }l\to\infty,
$$
contradicting \eqref{eqn:sn/Hn}.  So Property $(\star)$ holds.
\end{proof}

The hypotheses can be slightly weakened, to hold for ``almost all points'' in the array.

\begin{theorem}\label{thm:generalthm}
Let $\{b_{nj}\}$ be a triangular array of points in $[0,1]$.  Suppose, for any $\epsilon>0$, there exists $J\subset[0,1]$ given by a finite union of closed intervals of total length less than $\epsilon$, and positive constants $\alpha_1,\alpha_2,B_1,B_2$, such that 
\begin{enumerate}
\item $0<\alpha_2\leq \alpha_1<1+\alpha_2$.
\item There exists $n_0$ such that for every integer $n\geq n_0$ and $j\in\{1,\ldots,n\}$ such that $j/n\not\in J$, we have
$$
 B_1\Bigl|\frac{j-k}{n}\Bigr|^{\alpha_1} \leq |b_{nj}-b_{nk}| \leq B_2\Bigl|\frac{j-k}{n}\Bigr|^{\alpha_2}  \hbox{for all  } k\in\{1,\ldots,n\}. $$
 \end{enumerate}
Then Property ($\star$) holds for $\{b_{nj}\}$.  
\end{theorem}

\begin{proof}
Let the quantities $\{j_n\}_{n\in\NN}$, $\delta>0$ and the subsequence $\{n_l\}_{l\in1}^{\infty}$ be as in the first paragraph of the proof of Theorem \ref{thm:6}, with equation \eqref{eqn:sn/Hn} holding.  By possibly passing to a further subsequence we may assume that there exists $x\in[0,1]$ such that  $j_{n_l}/n_l\to x$.

Choose $J$ corresponding to $\epsilon=\delta/2$.  
If $x\not\in J$ then for sufficiently large $l$ we have $j_{n_l}/n_l\not\in J$, and the rest of the proof  goes through with no change to yield
 $$\frac{s_{n_l}(j_{n_l})}{H_{n_l}(j_{n_l})}\to 0 \hbox{ as } l\to\infty,$$ contradicting \eqref{eqn:sn/Hn}.

To get a contradiction in case $x\in J$, then without loss of generality (passing to a subsequence) we may suppose $j_{n_l}/n_l\in J$ for all $l$, and write 
$$\begin{aligned}
\frac{1}{H_{n_l}} = \frac{1}{n_l}\sum_{k=1}^{n_l}\frac{1}{|b_{n_lk}-m_{n_lj_{n_l}}|} & \ = \  \frac{1}{n_l}\sum_{\frac{k}{n_l}\in J} \frac{1}{|b_{n_lk}-m_{n_lj_{n_l}}|} \ + \  \frac{1}{n_l}\sum_{\frac{k}{n_l}\not\in J} \frac{1}{|b_{n_lk}-m_{n_lj_{n_l}}|} \\
&\ =:\  S_1 \ + \ S_2.
\end{aligned}$$
Using the length of $J$, we see that the sum $S_1$ consists of  $n_l\epsilon$ terms or fewer.  Also, $s_{n_l}(j_{n_l})\leq |b_{n_lk}-m_{n_lj_{n_l}}|$ for all $k$, so that 
$$S_1\leq \frac{\epsilon}{s_{n_l}(j_{n_l})}.$$  
To estimate $S_2$, we split the sum into two pieces:
$$
S_2 \ = \  \frac{1}{n_l}\sum_{\frac{k}{n_l}\not\in J, \frac{k}{n_l}>x } \frac{1}{|b_{n_lk}-m_{n_lj_{n_l}}|} \ + \   \frac{1}{n_l}\sum_{\frac{k}{n_l}\not\in J, \frac{k}{n_l}<x } \frac{1}{|b_{n_lk}-m_{n_lj_{n_l}}|}  \ =: \  T_1  +   T_2.
$$
To estimate $T_1$, we let $r_l$ be the smallest integer such that $r_l/n_l\not\in J$ and   $k\geq r_l$ for all $k\in\{1,\ldots,n_l\}$ such that $k/n_l>x$.  Then $r_l>j_{n_l}$ and so 
$$\begin{aligned}
T_1  \ = \   \frac{1}{n_l}\sum_{\frac{k}{n_l}\not\in J, \frac{k}{n_l}>x } \frac{1}{|b_{n_lk}-m_{n_lj_{n_l}}|} 
 \ & \leq \  \frac{1}{n_l}\sum_{\frac{k}{n_l}\not\in J, \frac{k}{n_l}\geq\frac{r_l}{n_l}  } \frac{1}{|b_{n_lk}-m_{n_lj_{n_l}}|} \\
&\leq\ \frac{1}{n_l}\left(\frac{2}{s_{n_l}(j_{n_l})} +   \sum_{\frac{k}{n_l}\not\in J, \frac{k}{n_l}>\frac{r_l+1}{n_l}  } \frac{1}{|b_{n_lk}-m_{n_lj_{n_l}}|} \right) \\
&\leq \ \frac{1}{n_l}\left(\frac{2}{s_{n_l}(j_{n_l})} +   \sum_{k=r_l+2 }^{n_l} \frac{2}{|b_{n_lk}-b_{n_lr_l}|} \right) \\
&\leq  \  \frac{1}{n_l}\left(\frac{2}{s_{n_l}(j_{n_l})} \  + \  \frac{2}{B_1}\int_{\frac{1}{n_l}}^1 \frac{dx}{x^{\alpha_1}}         \right) 
\end{aligned}$$
following the proof of Theorem \ref{thm:6}.  Using the same estimates of the integral as before, we obtain $s_{n_l}(j_{n_l})T_1\to 0$.  

A similar argument can be used to estimate $T_2$, to yield $s_{n_l}(j_{n_l})T_2\to 0$.  Hence $s_{n_l}(j_{n_l})S_2\to 0$ as $l\to\infty$. 

 Finally,
$$
\limsup_{l\to\infty}\frac{s_{n_l}(j_{n_l})}{H_{n_l}(j_{n_l})} \  = \  \limsup_{l\to\infty}\bigl(s_{n_l}(j_{n_l})S_1 +  s_{n_l}(j_{n_l})S_2\bigr)  \ \leq \  \epsilon + 0 = \delta/2 <\delta,
$$
contradicting \eqref{eqn:sn/Hn}.  
\end{proof}

\begin{figure}
\noindent\includegraphics[height=10cm]{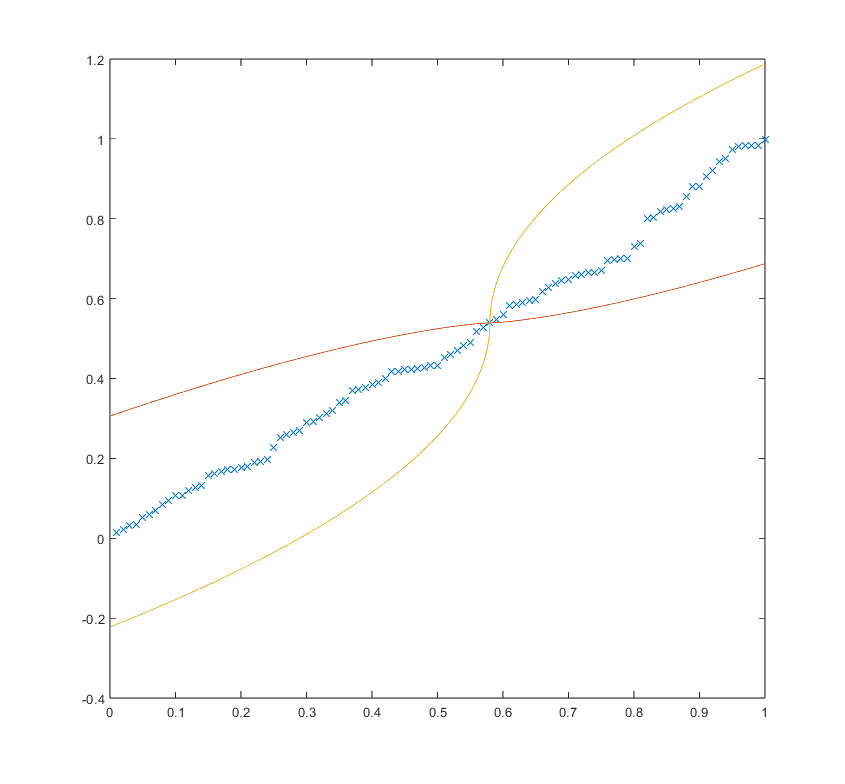}
\caption{Illustration of the bounds (valid for all $k$)   \\[10pt] 
\centerline{$B_1|(j-k)/n|^{\alpha_1} \leq |b_{nj}-b_{nk}| \leq B_2|(j-k)/n|^{\alpha_2}$ }\\[10pt]
when $j=58$, for a randomly generated distribution of 100 points in the interval $[0,1]$. 
Here $0.5<\alpha_2<1<\alpha_1<1.5$. 
 }
\end{figure}

\begin{figure}
\includegraphics[height=10cm]{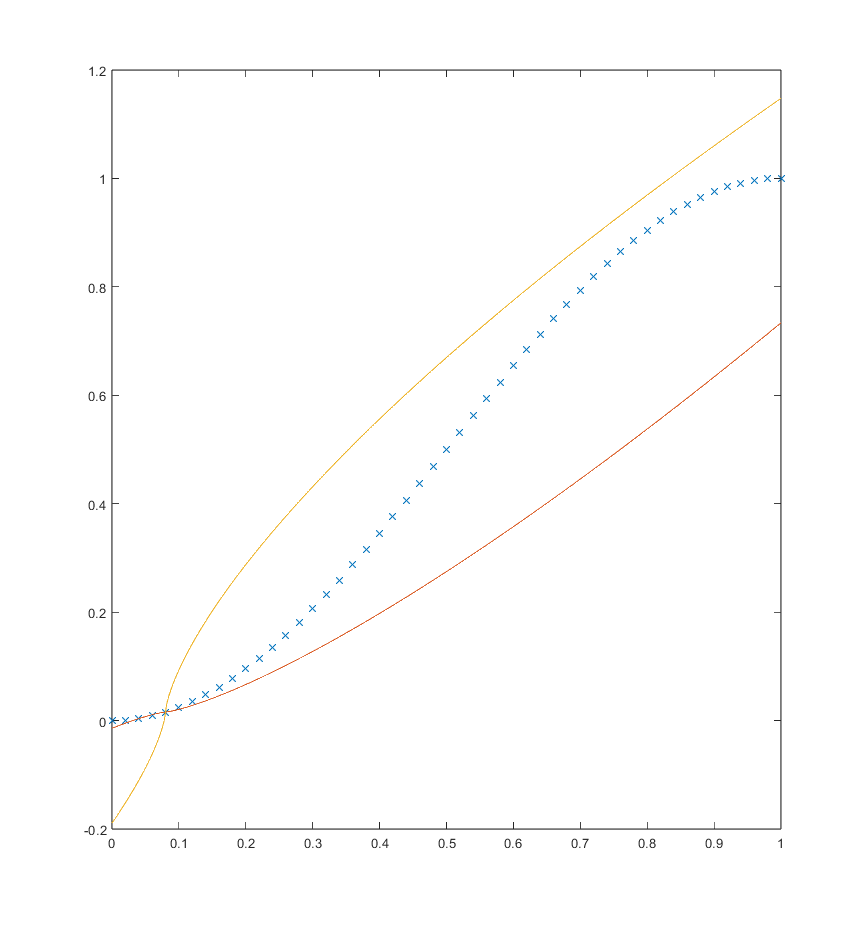}
\caption{Illustration of the estimates in Theorem \ref{thm:generalthm} for a 50-point approximation to the equilibrium distribution for $[0,1]$ with  $j=5$, and $0.5<\alpha_2<1<\alpha_1<1.5$. The lower estimate only holds outside a (small) neighborhood of the end points.   \\[10pt] 
  \\   
 }

\end{figure}

The conditions of the theorem are true for many arrays; a couple of illustrations of the estimates are given in Figures 1 and 2.  This includes arrays in which points converge to a reasonable   distribution over the interval.  In what follows, $\delta_{b_{nj}}$ denotes the discrete (Dirac) probability measure supported at $b_{nj}$ and $dx$ denotes Lebesgue measure.

\begin{corollary}
Let $\{b_{nj}\}$ be a triangular array of points in $[0,1]$, and suppose $$\frac{1}{n+1}\sum_{j=0}^n\delta_{b_{nj}} \longrightarrow \varphi'(x)dx \hbox{ as }n\to\infty$$ for some $C^1$  function $\varphi$ on $[0,1]$ with 
$\varphi'\geq 0$, where `$\longrightarrow$' denotes weak-$^*$ convergence of measures.  Suppose also that $\{x\in[0,1]\colon \varphi'(x) = 0\}$ is finite.  Then $\{b_{nj}\}$ has property ($\star$).
\end{corollary}

\begin{proof}
Let $J$ be a union of closed intervals of total length less than $\epsilon$ whose interior covers all the points $x$ where $\varphi'(x)=0$.  By continuity, $\varphi'>0$ on the closure of $[0,1]\setminus J$ and therefore attains a minimum $m$ and maximum $M$ on this set.  With this set $J$, the hypotheses of Theorem \ref{thm:generalthm} are satisfied (for sufficiently large $n$) when $B_1=m/2$, $B_2=2M$, and $\alpha_1=\alpha_2=1$. 
\end{proof}

\section{Approximation properties}\label{sec:approx}


Before stating our main theorem, we need to recall some basic notions and classical results.

Let $K\subset\CC$ be compact.  The \emph{$n$-th order diameter of $K$} is 
$$
d_n(K):=\sup\{ |VDM(a_1,\ldots,a_n)|^{\frac{2}{n(n+1)}}\colon \{a_1,\ldots,a_n\}\subset K\}
$$
where
$$
VDM(a_1,\ldots,a_n) := \det\begin{bmatrix}1&1&\cdots&1\\ a_1&a_2&\cdots&a_n\\
a_1^2&a_2^2&\cdots&a_n^2\\ \vdots&\vdots&\ddots&\vdots\\ a_1^{n-1}&a_2^{n-1}&\cdots&a_n^{n-1}  
\end{bmatrix}.
$$
is the Vandermonde determinant associated to the finite set $\{a_1,\ldots,a_n\}$.  

The \emph{transfinite diameter of $K$} is $$d(K):=\lim_{n\to\infty}d_n(K).$$  

Recall that a \emph{monic polynomial}  is of the form $z^n + (\hbox{terms of degree} <n)$, i.e., has leading coefficient 1. The $n$-th Chebyshev constant is 
$$
\tau_n(K):=\inf\{\|p\|_K\colon p \hbox{ is a monic polynomial of degree }\leq n\}^{1/n},
$$
and it is a classical result  that 
\begin{equation}\label{eqn:td} \lim_{n\to\infty}\tau_n(K)=d(K).  
\end{equation}
(See  e.g. Chapter 5 of \cite{ransford:potential}), where it is also shown that the transfinite diameter coincides with the potential-theoretic \emph{logarithmic capacity of $K$}.)

 Let $\calC=\{a_{nj}\}_{j=1,\ldots,n;n=1,2,\ldots}\subset K$ be an interpolation array. Given $n\in\NN$, define for each $j=1,\ldots,n$ the \emph{fundamental Lagrange interpolating polynomial}
$$\ell_j ^{(n)}(z):=\frac{VDM(a_{n1},\ldots,a_{n(j-1)},z,a_{n(j+1)},\ldots,a_{nn})}{VDM(a_{n1},\ldots,a_{nn})}.$$
This is the unique polynomial of degree $n-1$ that satisfies $\ell_j^{(n)}(a_{nj})=1$ and $\ell_j^{(n)}(a_{nk})=0$ if $k\neq j$.   The \emph{$n$-th Lebesgue constant} is 
$$
\Lambda_n:=\sup_{z\in K}\sum_{j=1}^n |\ell_j^{(n)}(z)|.
$$

We recall the following theorem relating the above notions to polynomial approximation on $K$.  Recall that `$\rightrightarrows$' denotes uniform convergence.

\begin{theorem}[\cite{bloomboschristensenlev:polynomial}, Theorem 1.5]   \label{thm:bbcl}
Let $K\subset\CC$ be a regular, polynomially convex, compact set. Consider the following four properties which an array $\{A_{nj}\}_{j=0,\ldots,n; n=1,2,\ldots}\subset K$ may or may not possess:
\begin{enumerate}
\item $\lim_{n\to\infty}\Lambda_n^{1/n} =1$;
\item $\lim_{n\to\infty} |VDM(A_{n0},\ldots,A_{nn})|^{\frac{2}{n(n+1)}}=d(K)$;
\item $\lim_{n\to\infty} \frac{1}{n+1}\sum_{j=0}^n\delta_{A_{nj}}=\mu_K$ weak$\,^*$;
\item $L_nf\rightrightarrows f$ on $K$ for each $f$ holomorphic on a neighborhood of $K$.
\end{enumerate}
Then (1) $\Rightarrow$ (2) $\Rightarrow$ (3) $\Rightarrow$ (4), and none of the reverse implications is true. \qed
\end{theorem}

\begin{remark}\rm A regular, polynomially convex compact set in $\CC$ is one whose complement consists of a single unbounded component, and whose potential-theoretic extremal function is continuous.  In particular, this is true for an  interval, and any  finite union of Jordan arcs or curves.   
\end{remark}

We can now prove our main theorem that relates Property (2) in Theorem \ref{thm:bbcl} to Property $(\star)$ from the previous section.

\begin{theorem}\label{thm:d}
Let $\fl=\{a_j\}_{j=1}^{\infty}$ denote the fast Leja points on the interval $I=[0,1]$.  Suppose Property $(\star)$  holds for $\fl$.  Then 
$$
\lim_{n\to\infty} |VDM(a_1,\ldots,a_n)|^{\frac{2}{n(n+1)}} = d(I).
$$
Hence $\fl$ is good for polynomial interpolation.
\end{theorem}

\begin{proof}
Let $$L_n:=VDM(a_1,\ldots,a_n)=\prod_{j\neq k} |a_j-a_k|.$$  
By definition, $\limsup_{n\to\infty} L_n^{\frac{2}{n(n+1)}} \leq d(I)$ so it  suffices to prove $\liminf_{n\to\infty} L_n^{\frac{2}{n(n+1)}} \geq d(I)$.  

Let $p_n(z)=\prod_{j=1}^{n-1}(z-a_j)$ for each integer $n>1$.   By a calculation,
$$
\dfrac{L_{n}}{L_{n-1}} = |p_n(a_n)|.
$$
Take $b\in I$ such that $\|p_n\|_I=|p_n(b)|$.  Then $b\in(a_{k_1},a_{k_2})$ where $a_{k_1},a_{k_2}$ are neighboring fast Leja points in $\fl_n$.  
Let $a_{k_1k_2}=\tfrac{1}{2}(a_{k_1}+a_{k_2})$ denote their midpoint.  Then applying Proposition \ref{prop:3} with  $a_{k_1k_2}$ translated to $0$, we have the estimate  
\begin{equation}\label{eqn:t5}
\frac{\|p_n\|_I}{|p(a_{k_1k_2})|}  = \frac{\|p_n\|_{[a_{k_1},a_{k_2}]}}{|p(a_{k_1k_2})|}     \leq 2\exp\left( |a_{k_1k_2}-a_{k_1}| \sum_{j=1}^{n}\frac{1}{|a_{k_1k_2}- a_j | }    \right).
\end{equation}
Note that $a_{k_1k_2}$ is a competitor for the next fast Leja point $a_n$, so $|p_n(a_{k_1k_2})|\leq |p_n(a_n)|$. Using this and \eqref{eqn:t5}, 
$$
\|p_n\|_I  \leq 2\exp\left( |a_{k_1k_2}-a_{k_1}| \sum_{j=1}^{n-1}\frac{1}{|a_{k_1k_2}- a_j | }    \right) |p_n(a_{k_1k_2})|
\leq 2\exp\left(n\frac{s_n(k_1)}{ H_{n}(k_1)}\right) |p_n(a_n)| .
 $$

Now $\|p_n\|_I\geq\tau_{n-1}(I)^{n-1}$, since $p_n$ is monic of degree $n-1$.  Hence
$$
\tau_{n-1}(I)^{n-1}\leq 2\exp\left(n\frac{s_{n}(k_1)}{ H_{n}(k_1)}\right)\frac{L_n}{L_{n-1}}.
$$
Let $\epsilon>0$.  Using property ($\star$), choose $N$ such that $\dfrac{s_n(j)}{H_n(j)}<\epsilon$ for all $n>N$ and $j<n$.  Then
$$\begin{aligned}
L_n &= \frac{L_n}{L_{n-1}}\frac{L_{n-1}}{L_{n-2}}\cdots\frac{L_{N+1}}{L_{N}}L_{N}  \\
& \geq \tau_{n-1}(I)^{n-1}\tau_{n-2}(I)^{n-2}\cdots\tau_{N+1}(I)^{N+1}
2^{n-N-1}\exp\left(-\epsilon \sum_{j=N+1}^n j \right)L_N.
\end{aligned}$$
In view of \eqref{eqn:td},  $\tau_j(I)\to d(I)$ as $j\to\infty$.  Hence the weighted geometric averages also converge:
$$
\left(\prod_{j=N+1}^n \tau_{j}(I)^{j}\right)^{1/\left(\sum j\right)} \longrightarrow d(I) \hbox{ as } n\to\infty.
$$
Also $$\displaystyle \sum_{j=N+1}^n j =\frac{(n+N+1)(n-N)}{2}=\frac{n^2}{2} +O(n).$$  So for any $n>N$ sufficiently large, 
$$
(L_n)^{\frac{2}{n(n+1)}} \geq (d(I)-\epsilon)^{\frac{n^2+O(n)}{n^2+n}}2^{\frac{2(n-N-1)}{n(n+1)}}
\exp\left( -\epsilon\frac{n^2+O(n) }{n^2+n}    \right) L_N^{\frac{1}{n(n+1)}}.
$$
Letting $n\to\infty$,
$$
\liminf_{n\to\infty}  (L_n)^{\frac{2}{n(n+1)}} \geq (d(I)-\epsilon) e^{-\epsilon}
$$
and since $\epsilon>0$ was arbitrary, $\liminf_{n\to\infty}  (L_n)^{\frac{2}{n(n+1)}}\geq d(I)$.

So $\lim_{n\to\infty} (L_n)^{\frac{2}{n(n+1)}} = d(I)$. Hence $\calF$ satisfies the second condition in Theorem \ref{thm:bbcl}, so is good for polynomial interpolation.
 \end{proof}

\begin{remark} \rm The above proof is based on the fact that the ratio $\dfrac{\|p_n\|_I}{|p(a_n)|}$ does not grow exponentially.  A sequence of points $\{a_1,a_2,\ldots\}$ for which the ratio has subexponential growth (where $p_n(z)=(z-a_1)\cdots(z-a_{n-1})$) is called a \emph{pseudo Leja sequence}.  Bia\l as-Ciez and Calvi defined pseudo Leja sequences in \cite{bc:pseudo} and asked whether fast Leja points give a pseudo Leja sequence. By the above proof, the answer will be yes if Property ($\star$) holds.
\end{remark}

\section{Final Remarks}

\begin{enumerate}[1.]
\item   The proof of Theorem \ref{thm:d} goes through with no modification if we replace the interval $I$ by a finite union of real closed intervals.  Start with an initial Leja set consisting of the end points of each interval, and an initial candidate set consisting of the midpoints of each interval, then run the fast Leja algorithm as usual.  It may be possible to use a similar proof when $K$ is a finite union of smooth arcs or curves. (Note that the results of Section \ref{sec:sup} are only proved for real points.)

\item 
  Numerical evidence suggests that in fact, the ratio $\dfrac{\|p_n\|_I}{|p(a_n)|}$ is uniformly bounded (by some constant $<2$, see Figure 3).  Points with this stronger property are called \emph{$\tau$-Leja points}  \cite{an:simple}: for $\tau\in(0,1)$, a sequence of points $\{z_0,z_1,\ldots\}\subset I$ is $\tau$-Leja if 
$$
\tau\|p_n\|_I\leq |p_n(z_n)|
$$
where $p_n(z)=(z-z_0)\cdots(z-z_{n-1})$. 

 Totik recently proved in \cite{totik:lebesgue} that $\tau$-Leja points on a set with positive logarithmic capacity satisfy $\lim_{n\to\infty}\Lambda_n^{1/n} = 1$, the strongest property in Theorem \ref{thm:bbcl}.

\medskip

\noindent{\bf Conjecture.} {\it Fast Leja points on an interval are $\tau$-Leja points for some $\tau\in(\frac{1}{2},1)$.  }

\end{enumerate}

\begin{figure}
\noindent\includegraphics[height=8cm]{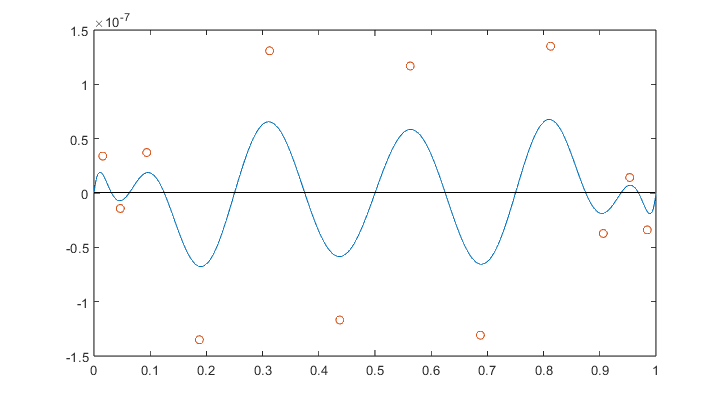}
\caption{A plot of the graph $(x,p_{13}(x))$, where the zeros of $p_{13}$ are the first 13 fast Leja points on $[0,1]$.  The points $(m_j,2p_{13}(m_j))$ are indicated by circles, where $m_j$ is the midpoint of the $j$-th interval between adjacent fast Leja points.  These midpoints are the candidates for the next fast Leja point.}
\end{figure}

\bibliographystyle{abbrv}
\bibliography{myreferences2024}

\begin{thebibliography}{1}

\bibitem{an:simple}
V.~Andrievskii and F.~Nazarov.
\newblock A simple upper bound for {L}ebesgue constants associated with {L}eja
  points on the real line.
\newblock {\em J. Approx. Theory}, 275:Paper No. 105699, 13pp., 2022.

\bibitem{bcr:fast}
J.~Baglama, D.~Calvetti, and L.~Reichel.
\newblock Fast {L}eja points.
\newblock {\em Electron. Trans. Numer. Anal.}, 7:124--140, 1998.

\bibitem{bc:pseudo}
L.~Bia{\l}as-Ciez and J.-P. Calvi.
\newblock Pseudo {L}eja sequences.
\newblock {\em Ann. Mat. Pura Appl. Series IV}, 191(1):53--75, 2012.

\bibitem{bloomboschristensenlev:polynomial}
T.~Bloom, L.~Bos, C.~Christensen, and N.~Levenberg.
\newblock Polynomial interpolation of holomorphic functions in $\mathbb{C}$ and
  $\mathbb{C}^n$.
\newblock {\em Rocky Mountain J. Math}, 22(2):441--470, 1992.

\bibitem{bdMsv:computing}
L.~Bos, S.~De{M}archi, A.~Sommariva, and M.~Vianello.
\newblock Computing multivariate {F}ekete and {L}eja points by numerical linear
  algebra.
\newblock {\em Siam. J. Numer. Anal.}, 48(5):1984--1999, 2010.

\bibitem{cl:uniform}
J.~{P}aul {C}alvi and N.~Levenberg.
\newblock Uniform approximation by discrete least squares polynomials.
\newblock {\em J. Approx. Theory}, 152:82--100, 2008.

\bibitem{ransford:potential}
T.~Ransford.
\newblock {\em Potential Theory in the Complex Plane}.
\newblock Cambridge University Press, Cambridge, 1995.

\bibitem{totik:lebesgue}
V.~Totik.
\newblock The {L}ebesgue constants for {L}eja points are subexponential.
\newblock {\em J. Approx. Theory}, 287:Paper No. 105863, 15pp., 2023.

\bibitem{mv:weakly}
M.~Vianello.
\newblock Webpage: \texttt{https://www.math.unipd.it/\~marcov/CAAwam.html}.

\end{thebibliography}

\end{document}